\documentclass{amsproc}

\usepackage[english]{babel}
\usepackage[latin1]{inputenc}
\usepackage{fancyhdr}
\usepackage{indentfirst}
\usepackage{amsthm}
\usepackage{graphicx}
\usepackage{newlfont}
\usepackage{amsmath}
\usepackage{amsfonts}
\usepackage{latexsym}
\usepackage{amssymb}
\usepackage{amscd}
\usepackage{amsfonts}
\usepackage{amsbsy}
\usepackage{pst-all}
\usepackage{verbatim}
\usepackage{hyperref}

\newtheorem{teo}{Theorem}[section]

\newtheorem{lem}[teo]{Lemma}
\newtheorem{defin}[teo]{Definition}
\newtheorem{ex}[teo]{Example}

\def \R {\mathbb {R}}
\def \RN {\mathbb{R}^n}
\def \RNN {\mathbb{R}^N}
\def \de {\partial}
\def \G {\mathbb{G}}
\def \g {\mathfrak{g}}
\def \z {\mathfrak{z}}
\def \b {\mathfrak{b}}
\def \Lop {\mathcal{L}}
\def \e {\varepsilon}
\def \L {\Lambda}
\def \l {\lambda}
\def \dl {\delta_\l}
\def \xgrad {\nabla_X}
\def \buz {B_1(0)}
\def \kom {\mathcal{K}_{\Omega}}
\def \kA {\kom^A}
\def \H {\mathbb{H}}
\def \ball {B_R(x_0)}
\def \dball {B_{2R}(x_0)}
\def \hball {B_{\frac{R}{2}}(x_0)}
\def \mlL {M_m(\l,\L)}
\def \BmO {B_{2K}(x)\smallsetminus O}

\newcommand{\Trac}{\mathop{\mathrm{Tr}}} 

\numberwithin{equation}{section}

\title[A certain critical density property]
{A certain critical density property\\for invariant Harnack inequalities in H-type groups}
\author[G. Tralli]{Giulio Tralli}
\address{Dipartimento di Matematica\\Universit\`a di Bologna\\
Piazza di Porta San Donato, 5\\ IT-40126 Bologna\\ Italy}
\email{giulio.tralli2@unibo.it}


\begin{document}

\maketitle 

\section*{Abstract} We consider second order linear degenerate-elliptic operators which are elliptic with respect to horizontal directions generating a stratified algebra of H-type. Extending a result in \cite{GT} for the Heisenberg group, we prove a critical density estimate by assuming a condition of Cordes-Landis type. We then deduce an invariant Harnack inequality for the non-negative solutions from a result in \cite{FGL}.

\section{Introduction}

The Lie algebras of type H were introduced by Kaplan in \cite{K}. Arising from the so-called compositions of quadratic forms, they are a class of step two Carnot algebras generalizing the Heisenberg-Weyl algebra. We consider an homogeneous Lie group $\G$ whose algebra is of H-type. In this setting, we fix a basis $X_1,\ldots,X_m$ for the horizontal layer such that it is orthonormal with respect to a peculiar inner product. If $\Omega\subseteq\G$ is an open set, the operators we are going to deal with have the following non-divergence form
\begin{equation}\label{uniXell} 
 \Lop_A=\sum_{i,j=1}^ma_{ij}(x)X_iX_j\qquad\,\,\,\,\mbox{for  }x\in\Omega.
\end{equation}
The matrix $A(x)=(a_{ij}(x))_{i,j=1}^m$ is supposed to be symmetric and uniformly positive definite in $\Omega$. This means there exist $0<\l\leq\L$ such that, for any $x$, we have $$\l\left\|\xi\right\|^2\leq\left\langle A(x)\xi,\xi\right\rangle\leq\L\left\|\xi\right\|^2$$
for every $\xi\in\RNN$. We denote by $\mlL$ the set of the $m\times m$ symmetric matrices satisfying these bounds. We assume that the coefficients $a_{ij}$ are smooth functions, but the estimates we establish are independent of their regularity. The aim of this paper is to prove an invariant Harnack inequality for such degenerate elliptic operators by assuming a Cordes-type condition as
\begin{equation}\label{CordesLandis}
\sup_{x\in\Omega}{\left(\frac{\Trac(A(x))+(Q+2-m)\max_{\left\|\xi\right\|=1}{\left\langle A(x)\xi,\xi\right\rangle}}{\min_{\left\|\xi\right\|=1}{\left\langle A(x)\xi,\xi\right\rangle}}\right)}<Q+4
\end{equation}
where $Q$ is the homogeneous dimension of the group. This condition implies the existence of a positive number $\delta$ (less than $2$) such that
\begin{equation}\label{deltaLandis}
\Trac(A(x))+(Q+2-m)\max_{\left\|\xi\right\|=1}{\left\langle A(x)\xi,\xi\right\rangle}\leq (Q+4-\delta)\min_{\left\|\xi\right\|=1}{\left\langle A(x)\xi,\xi\right\rangle}\,\,\,\,\,\forall\, x\in\Omega.
\end{equation}
The constants appearing in the Harnack inequality for $\Lop_A$ will depend on $A$ just through the constants $\l,\L,\delta$ and the structure of $\G$. We would like to stress that the assumption 
$$\frac{\L}{\l}<\frac{Q+3}{Q+1}$$
implies the previous condition with $\delta=(Q+1)\left(\frac{Q+3}{Q+1}-\frac{\L}{\l}\right)$ for any $A\in\mlL$. Cordes-type estimates in subelliptic settings for operators in non-divergence form are already present in the literature. They have been considered for the problem of interior regularity of $p$-harmonic functions in the Heisenberg group in \cite{DM} and in the Gru{\v{s}}in plane in \cite{DDFM}. The original Cordes condition is different from (\ref{CordesLandis}). Actually, ours is closer in the aspect and in the purposes to the one used by Landis in \cite{L} (see also \cite{La}, Chapter 1, Section 7). That is why we will refer to (\ref{deltaLandis}) as the $\delta$-Landis condition.

In order to get the Harnack inequality, we exploit the axiomatic procedure developed by Di Fazio, Guti\'errez and Lanconelli in \cite{FGL}. The key tools are the \textit{critical density} and the \textit{double-ball property}. These properties, hidden in \cite{KrSa}, arose from the techniques in \cite{C} for uniformly elliptic fully non-linear equations. These notions were then developed for the linearized Monge-Ampère equation in \cite{CG}, where Caffarelli and Gutiérrez proved an invariant Harnack inequality on some suitable sets. For an exhaustive study of these facts, we refer the reader to \cite{G}. The approach in \cite{FGL} can be applied in very general settings and in particular in the setting of homogeneous Carnot groups. As a matter of fact, it has been exploited in \cite{GT} in the Heisenberg group $\H$. For operators as in (\ref{uniXell}) where the $X_j$'s are the usual generators of $\H$, Gutiérrez and Tournier proved the double ball property and, assuming that the matrix $A$ has eigenvalues sufficiently close to $1$, the critical density. In \cite{Tr} we have already investigated the notion of the double ball property and we have proved that it holds true in any Carnot group of step two: in particular, this can be applied in H-type groups.\\
To prove the critical density in our settings, we follow the main points of the arguments introduced by Gutiérrez and Tournier. The key point is the existence of a specific barrier function for $\Lop_A$. Our proof of this fact is different from the one in \cite{GT} and it relies more on the structural properties of the Carnot groups in general and of the H-type algebras in particular. Denoting by $\ball$ some particular homogeneous balls of radius $R$ centered at $x_0$ we will define, our main result is thus the following invariant Harnack inequality.

\begin{teo}\label{main} Suppose $0<\l\leq\L$ and $0<\delta\leq2$. There exist constants $C$ and $\eta$ depending just on $\L,\l,\delta$ and the structure of $\G$ such that, for any $A\in\mlL$ satisfying the $\delta$-Landis condition, if we have a function $u$ with
$$u\geq 0\,\,\,\,\,\mbox{and }\,\,\Lop_A u=0\,\,\,\,\,\mbox{in }B_{\eta R}(x_0)\subset\Omega,$$
then it has to be
$$\sup_{\ball}{u}\leq C\inf_{\ball}{u}.$$
\end{teo}

\noindent Once more we stress that, if we suppose $\frac{\L}{\l}<\frac{Q+3}{Q+1}$, we have an invariant Harnack inequality which is uniform in the class of $A\in\mlL$.

This paper is organized as follows. In Section \ref{defprel}, we first give the main definitions we are going to exploit. Then, we state the theorem about the critical density and we show how Theorem \ref{main} can be deduced from it. In Section \ref{mais}, we prove the critical density estimate. In Section \ref{excom}, we will show some explicit examples and the importance of certain assumptions.

\textbf{Acknowledgments} We do wish to thank Ermanno Lanconelli and Cristian Gutiérrez for the advices and the discussions about the topic.

\section{Definitions and preliminaries}\label{defprel}

Di Fazio, Gutiérrez and Lanconelli presented their axiomatic approach in the abstract setting of doubling quasi metric H\"older space (see \cite{FGL}, Definition 2.2-2.3). Before fixing the specific setting we deal with, we would like to state the critical density property in its abstract form. Following the notations in \cite{FGL}, we fix an H\"older continuous quasi-distance $d$ and a positive measure $\mu$ on a $\sigma$-algebra containing the $d$-balls. We denote by $\kom$ a family of $\mu$-measurable functions with domain contained in an open set $\Omega$. If $u\in\kom$ and its domain contains a set $V\subset\Omega$, we write $u\in\kom(V)$.

\begin{defin}\label{ecd}
We say that $\kom$ satisfies the critical density property if there exist $0<\e<1$ and $c>0$ such that, for every $\dball\subset\Omega$ and for every $u\in\kom(\dball)$ with
$$\mu(\{x\in\ball\,:\,u(x)\geq 1\})\geq \e\mu(\ball),$$
we have
$$\inf_{\hball}{u}\geq c.$$
\end{defin}
The homogeneous Lie groups are a remarkable example of doubling quasi metric H\"older space (\cite{FGL}, Remark 2.5). We can indeed fix as $\mu$ the Lebesgue measure, denoted simply by $\left|\cdot\right|$ in what follows, which is the left and right-invariant Haar measure in such groups (see e.g. \cite{BLU}, Proposition 1.3.21). Moreover, we can choose an homogeneous symmetric norm (see \cite{BLU}, Definition 5.1.1) inducing a quasi-distance $d$. For the procedure we want to exploit, they are needed also a ring condition (\cite{FGL}, Definition 2.6) and a reverse doubling condition: these properties are easily satisfied in the homogenous Lie groups. That is why it is worthwhile to investigate the critical density and the double ball property in such settings.

We now consider an homogenous Lie group $\G=(\RNN,\circ,\dl)$. We suppose that its Lie algebra $\g$ is of type H. Let us give the definition.

\begin{defin}\label{Hdef} An H-type algebra is a finite-dimensional real Lie algebra $(\g,[\cdot,\cdot])$ which can be endowed with an inner product $\left\langle \cdot,\cdot\right\rangle$ such that
$$[\z^{\bot},\z^\bot]=\z,$$
where $\z$ is the center of $\g$. Moreover, for any fixed $z\in\z$, the map $J_z:\z^\bot\longrightarrow\z^\bot$ defined by
$$\left\langle J_z(v),w\right\rangle=\left\langle z,[v,w]\right\rangle\qquad\forall\, w\in\z^\bot$$
is an orthogonal map whenever $\left\langle z,z\right\rangle=1$. We say that a simply connected Lie group is an H-type group if its Lie algebra is an H-type algebra.
\end{defin}
The H-type groups are particular Carnot groups of step two: a stratification is just given by
$$\g=\z^\bot\oplus\z.$$
We are going to denote $\b=\z^\bot$ and $\left\|q\right\|=\left\langle q,q\right\rangle$ for $q\in\g$. Moreover, we put $m=dim(\b)$ and $n=dim(\z)$. The associated homogeneity in the Lie group is thus given by the dilations $\dl((x_{(1)},x_{(2)}))=(\l x_{(1)},\l^2 x_{(2)})$ (for $(x_{(1)},x_{(2)})\in\RNN=\R^m\times\RN$) and the homogeneous dimension is equal to $Q:=m+2n$. \\
We now fix an orthonormal (with respect to $\left\langle \cdot,\cdot\right\rangle$) basis $X_1,\ldots,X_m$ for $\b$ and an orthonormal basis $Z_1,\ldots,Z_n$ for $\z$. Then $X_1,\ldots,X_m,Z_1,\ldots,Z_n$ is an orthonormal basis for $\g$ and we have
$$v+z=\sum_{j=1}^m\left\langle v,X_j\right\rangle X_j+\sum_{k=1}^n\left\langle z,Z_k\right\rangle Z_k\qquad\forall\, v\in\b,\,\,\,\forall\, z\in\z.$$
For any $z\in\z$, the map $J_z$ satisfies, among the others, the following properties
\begin{equation}\label{jayz}
\left\langle J_z(v),v\right\rangle=0\qquad\mbox{and}\qquad\left\|J_z(v)\right\|=\left\|z\right\|\left\|v\right\|\qquad\forall\, v\in\b
\end{equation}
which we will use in the following Section. A proof of these facts and other nice properties of H-type groups can also be found in \cite{DGN} (Section 6) and in \cite{BLU} (Chapter 18).

\begin{ex} As it is well-known, the Heisenberg-Weyl group is a particular H-type group. If $(x_1,\ldots,x_{2k},z)\in\R^{2k+1}$ and the vector fields
$$X_j=\de_{x_j}-\frac{x_{j+k}}{2}\de_{z},\,\,\,\,\,\,\, X_{j+k}=\de_{x_{j+k}}+\frac{x_j}{2}\de_{z}\,\,\,\,\,\,\,\mbox{(for } j\in\{1,\ldots,k\}),\,\,\,\,\,\,\, Z=Z_1=\de_{z}$$
are the usual generators of the Lie algebra, the standard inner product induced by the basis $\{X_1,\ldots,X_{2k},Z\}$ is the one needed for satisfying Definition \ref{Hdef}. Furthermore, in this case the map $J_{(\cdot)}$ is given by
$$J_{cZ}\left(\sum_{j=1}^{2k}a_jX_j\right)=c\sum_{j=1}^k(-a_{j+k}X_j+a_jX_{j+k}).$$
\end{ex}

Following Kaplan's notations, we define the functions $v:\G\longrightarrow\b$ and $z:\G\longrightarrow\z$ by the following relation
$$x=Exp(v(x)+z(x)),\qquad\mbox{for }x\in\G,$$
where $Exp$ denotes the exponential map. We remind that in this situation $Exp$ is a globally defined diffeomorphism with inverse denoted by $Log$. Thus, for any $x\in\G$ we have
$$v(x):=\sum_{j=1}^m\left\langle Log(x),X_j\right\rangle X_j,\qquad\mbox{and}\qquad z(x):=\sum_{k=1}^m\left\langle Log(x),Z_k\right\rangle Z_k.$$
The approach we want to use requires the choice of an homogeneous symmetric norm in order to define a quasi distance $d$ and the $d$-balls. In the H-type groups there are some preferable choices. As a matter of fact, let us consider the function
$$\varphi(x)=(\left\|v(x)\right\|^4+16\left\|z(x)\right\|^2)^{\frac{2-Q}{4}}.$$
Kaplan proved (\cite{K}, Theorem 2) that there exists a positive constant $k$ such that $k\varphi$ is the fundamental solution at the origin of the sub-Laplacian $\Delta_\G=\sum_{j=1}^mX_j^2$. Thus, let us choose as homogeneous symmetric norm the function 
$$d(x):=\left\{\begin{array}{cc}
(\varphi(x))^{\frac{1}{2-Q}} & \,\mbox{if }\,x\neq 0  \\
0 &	\,\mbox{if }\,x=0
\end{array}\right.,$$
which is a $\Delta_\G$-gauge (see \cite{BLU}, Definition 5.4.1 and Proposition 5.4.2). Danielli, Garofalo and Nhieu proved in \cite{DGN} (Theorem 6.8) that $d$ has also the remarkable property to be horizontally-convex. A direct proof of the fact that $d$ is an homogeneous symmetric norm can be found in \cite{BLU} (Remark 18.3.2). The quasi-distance is then given by $d(\xi,x):=d(x^{-1}\circ\xi)$, for any $\xi,x\in\G$, and the $d$-ball of radius $R$ centered at $x_0$ is
$$B_R(x_0)=x_0\circ B_R(0)=x_0\circ\{x\in\G\,:\,\left\|v(x)\right\|^4+16\left\|z(x)\right\|^2<R^4\}.$$
These balls have a great importance in the analysis of the sub-Laplacian and more in general in the geometry of $\G$. As a matter of fact, it is very well-known that a mean-value representation formula holds true on such balls (see e.g. \cite{BLU}, Theorem 5.5.4). Since we are going to use in the next Section the kernel $\psi_0$ of the surface mean-value formula at the origin, for the reader convenience we remind that
\begin{equation}\label{somean}
\psi_0(\xi)=\frac{\left\|\xgrad d(\xi)\right\|^2}{\left\|\nabla d(\xi)\right\|}\qquad\mbox{and}\qquad\frac{\beta}{R^{Q-1}}\int_{\de B_R(0)}\psi_0(\xi)\,d\sigma(\xi)=1
\end{equation}
for some positive constant $\beta$. Here we have denoted by $d\sigma$ the $N-1$-dimensional Hausdorff measure, by $\nabla d$ the gradient of the function $d$ and by $\xgrad d$ its horizontal gradient that is $\xgrad d(\xi)=(X_1d(\xi),\ldots,X_m d(\xi))$ (with an abuse of notations $\left\|\cdot\right\|$ is even the euclidean norm in $\R^m$ and in $\RNN$).

We will prove a critical density estimate in $\G$ for horizontally elliptic operators as in (\ref{uniXell}). To this aim, we have to say which is the set of Lebesgue measurable functions satisfying Definition \ref{ecd}. Keeping in mind the case of uniformly elliptic operators (see \cite{G}, Theorem 2.1.1) and the work by Gutiérrez and Tournier, if $A\in\mlL$ for some $0<\l\leq\L$ we put
\begin{equation}\label{kk}
\kA:=\{u\in C^2(V,\R) : V\subset\Omega, u\geq 0 \mbox{ and } \Lop_A u\leq 0 \mbox{ in } V\}.
\end{equation}
We stress that $\kA$ is closed under multiplications by positive constants as it is required in \cite{FGL} (for example in Theorem 4.7). In the next Section we are going to prove the following theorem.

\begin{teo}\label{cecd}
For every $0<\l\leq\L$ and $0<\delta\leq2$, the family $\kA$ satisfies the critical density property with respect to Definition \ref{ecd} for any $A\in\mlL$ satisfying the $\delta$-Landis condition. The constants $\e$ and $c$ depend on $A$ only through the constants $\l,\L,\delta$.
\end{teo}

Once we have proved this theorem, the proof of Theorem \ref{main} follows straightforwardly from the results in \cite{FGL} and in \cite{Tr}.

\begin{proof}[Proof of Theorem \ref{main}] In the doubling quasi metric H\"older space $(\G,d,\left|\cdot\right|)$ the reverse doubling condition and the ring condition are satisfied. Moreover, if $A\in\mlL$ for some $0<\l\leq\L$, the family $\kA$ satisfies the double ball property by the result in \cite{Tr}. If in addition $A$ satisfies the condition (\ref{deltaLandis}) for some positive $\delta$, the previous theorem tells us that also the critical density is satisfied by $\kA$. By the results in \cite{FGL} (Theorem 4.7, set of conditions (A), and Theorem 5.1), the functions belonging to the following subset of $\kA$
$$\{u\in C^2(V,\R) : V\subset\Omega, u\geq 0 \mbox{ and } \Lop_A u= 0 \mbox{ in } V\}$$
satisfy the invariant Harnack inequality. Since the constants appearing in the critical density and in the double ball property depend on $A$ only through $\l,\L$ and $\delta$, even the constants $C$ and $\eta$ in the Harnack inequality depend just on $\l,\L,\delta$ and the setting $\G$. In particular, they are independent of the regularity of the coefficients of the matrix $A$.
\end{proof}

\section{The main Lemma}\label{mais}

In order to prove Theorem \ref{cecd}, we follow the steps of the proof adopted in \cite{GT}. The crucial point is the existence of a very specific barrier function. To show this fact, let us compute explicitly the horizontal gradient $\xgrad d$ and the horizontal Hessian matrix $(X_iX_j d)_{i,j=1}^m$. We put $\phi=d^4$, which is a smooth function in the whole $\G$. If we define for any fixed $x\in\G$ the functions
$$\phi_j(t)=\phi(x\circ Exp(tX_j))\qquad\mbox{and}\qquad\phi_{i,j}(s,t)=\phi(x\circ Exp(sX_i)\circ Exp(tX_j))$$
for $s,t\in\R$ and $i,j\in\{1,\ldots,m\}$, by definition we have 
$$X_j\phi(x)=\phi_j'(0)\qquad\mbox{and}\qquad X_iX_j\phi(x)=\frac{\de^2}{\de s\de t}\phi_{i,j}(0,0).$$
We remind that the Campbell-Hausdorff formula for step two Lie algebras state that
$$Exp(A)\circ Exp(B)=Exp\left(A+B+\frac{1}{2}[A,B]\right)\qquad\forall\, A,B\in\g.$$
Thus, since $z(x)\in\z$, we get
\begin{eqnarray*}
&\,&Exp\left(\vphantom{\sum}v(x\circ Exp(sX_i)\circ Exp(tX_j))+z(x\circ Exp(sX_i)\circ Exp(tX_j))\right)=\\
&=&x\circ Exp(sX_i)\circ Exp(tX_j)=\\
&=&Exp(v(x)+z(x))\circ Exp\left(sX_i+tX_j+\frac{st}{2}[X_i,X_j]\right)=\\
&=&Exp\left(v(x)+z(x)+sX_i+tX_j+\frac{st}{2}[X_i,X_j]+\frac{s}{2}[v(x),X_i]+\frac{t}{2}[v(x),X_j]\right).
\end{eqnarray*}
Since $v(x)+sX_i+tX_j\in\b$ and $z(x)+\frac{st}{2}[X_i,X_j]+\frac{s}{2}[v(x),X_i]+\frac{t}{2}[v(x),X_j]\in\z$, we deduce
$$v(x\circ Exp(sX_i)\circ Exp(tX_j))=v(x)+sX_i+tX_j\qquad\mbox{and}$$
$$z(x\circ Exp(sX_i)\circ Exp(tX_j))=z(x)+\frac{st}{2}[X_i,X_j]+\frac{s}{2}[v(x),X_i]+\frac{t}{2}[v(x),X_j].$$
For $s=0$, this gives also
$$v(x\circ Exp(tX_j))=v(x)+tX_j\qquad\mbox{and}\qquad z(x\circ Exp(tX_j))=z(x)+\frac{t}{2}[v(x),X_j].$$
Now we have an explicit expression for $\phi_j(t)$ and $\phi_{i,j}(s,t)$. Straightforward calculations show that
$$X_j\phi(x)=4\left\langle \left\|v(x)\right\|^2v(x)+4J_{z(x)}(v(x)), X_j\right\rangle\qquad\mbox{and}$$
\begin{eqnarray}\label{Hessphi}
X_iX_j\phi(x)&=&4\left\|v(x)\right\|^2\delta_{ij}+8\left\langle v(x),X_i\right\rangle\left\langle v(x),X_j\right\rangle+16\left\langle z(x),[X_i,X_j]\right\rangle+\nonumber \\
&+&8\sum_{k=1}^n\left\langle J_{Z_k}(v(x)),X_i\right\rangle\left\langle J_{Z_k}(v(x)),X_j\right\rangle.
\end{eqnarray}
By the first equality and the first property in (\ref{jayz}), we get the relation $\left\|\xgrad\phi(x)\right\|^2=16\left\|v(x)\right\|^2\phi(x)$, which implies
\begin{equation}\label{xgradd}
\left\|\xgrad d(x)\right\|^2=\frac{\left\|v(x)\right\|^2}{d^2(x)}.
\end{equation}
The last relation had already been proved in \cite{DGN} (Lemma 6.3). On the other hand, for $i,j\in\{1,\ldots,m\}$ we get
\begin{eqnarray*}
X_iX_jd(x)&=&-\frac{3}{d(x)}X_id(x)X_jd(x)+\frac{1}{4d^3(x)}X_iX_j\phi(x)=\\
&=&\frac{\left\|\xgrad d(x)\right\|^2\delta_{ij}-3X_id(x)X_jd(x)}{d(x)}+\frac{2}{d^3(x)}\left(\left\langle v(x),X_i\right\rangle\left\langle v(x),X_j\right\rangle\vphantom{\sum_{k=1}^n}\right.\hspace{-0.11cm}+\\
&+&\left.2\left\langle z(x),[X_i,X_j]\right\rangle+\sum_{k=1}^n\left\langle J_{Z_k}(v(x)),X_i\right\rangle\left\langle J_{Z_k}(v(x)),X_j\right\rangle\right).
\end{eqnarray*}
If $A\in\mlL$, since the matrix $(\left\langle z(x),[X_i,X_j]\right\rangle)_{i,j=1}^m$ is skew-symmetric and the product of a symmetric matrix with a skew-symmetric one has zero trace, we have
\begin{eqnarray}\label{Lad}
\Lop_Ad(x)&=&\frac{1}{d(x)}\left\langle \left(\Trac(A(x))\mathbb{I}_m-3A(x)\right)\xgrad d(x),\xgrad d(x)\right\rangle+\nonumber\\
&+&\frac{2}{d^3(x)}\left(\left\langle A(x)V(x),V(x)\right\rangle+\sum_{k=1}^n\left\langle A(x)J_kV(x),J_kV(x)\right\rangle\right),
\end{eqnarray}
where we have denoted the vectors of $\R^m$ $(\left\langle v(x),X_j\right\rangle)_{j=1}^m$ and $(\left\langle J_{Z_k}(v(x)),X_j\right\rangle)_{j=1}^m$ respectively by $V(x)$ an $J_kV(x)$. We are now ready to prove our main Lemma, which is the counterpart of Lemma 3.1 in \cite{GT}. This is also the only part where the Cordes-Landis condition is needed.

\begin{lem}\label{BAR} Fix $0<\l\leq\L$. For any fixed $0<\delta\leq 2$, put $\alpha=Q-\delta$. For any open set $O$ such that $O\subseteq B_1(0)$, we consider the function
$$h(x)=-\frac{1}{\alpha}\int_{O}\left(d(x^{-1}\circ\xi)\right)^{-\alpha}\,d\xi.$$
For $\e>0$, let $\eta_{\e}\in C^{\infty}([0,+\infty))$ such that $0\leq\eta_{\e}\leq1$, $\eta_{\e}(\rho)=0$ for $0\leq\rho<\e$ and $\eta_{\e}(\rho)=1$ for $\rho\geq2\e$. Consider the $C^{\infty}$ function
$$h_{\e}(x)=-\frac{1}{\alpha}\int_{O}\frac{\eta_{\e}(d(x^{-1}\circ\xi))}{d^\alpha(x^{-1}\circ\xi)}\,d\xi$$
which converges uniformly to $h$ as $\e\rightarrow0^+$. Then, for any compact set $O'\subset O$, we have 
\begin{equation}\label{compop}
\Lop_A h_{\e}(x)\geq C\l\qquad\,\,\forall x\in O',
\end{equation}
for every $A\in\mlL$ satisfying the $\delta$-Landis condition and for every $0<2\e<d(O',\de O):=\inf{\{d(\xi^{-1}\circ x)\,:\,x\in O',\xi\in\de O\}}$. The positive constant $C$ has to depend just on $\delta,d$ and the $X_j$'s.
\end{lem}
\begin{proof} Put $g(\xi)=-\frac{1}{\alpha}d^{-\alpha}(\xi)$ and $g_\e(\xi)=g(\xi)\eta_{\e}(d(\xi))$. By the symmetry of $d$, these functions are symmetric, i.e. $g(\xi^{-1})=g(\xi)$ and $g_\e(\xi^{-1})=g_\e(\xi)$. Thus, we have
$$h_{\e}(x)=\int_{O}g_\e(x^{-1}\circ\xi)\,d\xi=\int_{O}g_\e(\xi^{-1}\circ x)\,d\xi.$$
We note that, for $x\in\buz$, we have $\buz\subseteq B_{2K}(x)$ where $K\geq 1$ is the constant appearing in the pseudo-triangle inequality for $d$. The smoothness of $g_\e$ and the left-invariance of the vector fields imply, for every $i,j=1,\ldots,m$, that
\begin{eqnarray*}
X_iX_jh_{\e}(x)&=&\int_{O}\left(X_iX_jg_\e(\xi^{-1}\circ\cdot)\right)(x)d\xi=\int_{O}\left(X_iX_jg_\e\right)(\xi^{-1}\circ x)d\xi=\\
&=&\int_{B_{2K}(x)}X_iX_jg_\e(\xi^{-1}\circ x)d\xi-\int_{\BmO}X_iX_jg_\e(\xi^{-1}\circ x)d\xi=\\
&=&\int_{B_{2K}(0)}X_iX_jg_\e(\xi^{-1})d\xi-\int_{\BmO}X_iX_jg_\e(\xi^{-1}\circ x)d\xi.
\end{eqnarray*}
Since $d\xi$ is also inversely invariant and the balls are symmetric, we get
\begin{eqnarray*}
X_iX_jh_{\e}(x)&=&\int_{B_{2K}(0)}X_iX_jg_\e(\xi)d\xi-\int_{\BmO}X_iX_jg_\e(\xi^{-1}\circ x)d\xi=\\
&=&\int_{\de B_{2K}(0)}X_jg_\e(\xi)\frac{X_id(\xi)}{\left\|\nabla d(\xi)\right\|}d\sigma(\xi)-\int_{\BmO}X_iX_jg_\e(\xi^{-1}\circ x)d\xi,
\end{eqnarray*}
where the last equality is justified by the divergence theorem: the vector fields $X_i$'s are indeed divergence-free because of the $\dl$-homogeneity (see \cite{BLU}, Remark 1.3.7). Assume now $0<\e<1$. Then, if $\xi$ belongs to a small open neighborhood of $\de B_{2K}(0)$, $g_\e(\xi)=g(\xi)$. Moreover, let $O'$ be a compact subset of $O$ such that $0<2\e<d(O',\de O)$. If $x\in O'$ and $\xi\in\BmO$, then $g_\e(\xi^{-1}\circ x)=g(\xi^{-1}\circ x)$. Thus, for $x\in O'$, we get
\begin{eqnarray*}
X_iX_jh_{\e}(x)\hspace{-0.24cm}&=&\int_{\de B_{2K}(0)}X_jg(\xi)\frac{X_id(\xi)}{\left\|\nabla d(\xi)\right\|}d\sigma(\xi)-\int_{\BmO}\hspace{-0.12cm}{X_iX_jg(\xi^{-1}\circ x)d\xi}=\\
&=&\int_{\de B_{2K}(0)}\frac{X_jd(\xi)X_id(\xi)}{(2K)^{\alpha+1}\left\|\nabla d(\xi)\right\|}d\sigma(\xi)-\int_{\BmO}\frac{X_iX_jd(\xi^{-1}\circ x)}{\left(d(\xi^{-1}\circ x)\right)^{\alpha+1}}d\xi+\\
&+&(\alpha+1)\int_{\BmO}\frac{X_id(\xi^{-1}\circ x)X_jd(\xi^{-1}\circ x)}{\left(d(\xi^{-1}\circ x)\right)^{\alpha+2}}d\xi.
\end{eqnarray*}
Hence, for $A\in\mlL$ and for $x\in O'$, we have
\begin{eqnarray*}
\Lop_A h_{\e}(x)&=&\int_{\de B_{2K}(0)}\frac{\left\langle A(x)\xgrad d(\xi),\xgrad d(\xi)\right\rangle}{(2K)^{\alpha+1}\left\|\nabla d(\xi)\right\|}d\sigma(\xi)+\\
&-&\int_{\BmO}\frac{\sum_{i,j=1}^m a_{ij}(x)X_iX_jd(\xi^{-1}\circ x)}{\left(d(\xi^{-1}\circ x)\right)^{\alpha+1}}d\xi+\\
&+&(\alpha+1)\int_{\BmO}\frac{\left\langle A(x)\xgrad d(\xi^{-1}\circ x),\xgrad d(\xi^{-1}\circ x)\right\rangle}{\left(d(\xi^{-1}\circ x)\right)^{\alpha+2}}d\xi\geq\\
&\geq&\frac{\l}{(2K)^{\alpha+1}}\int_{\de B_{2K}(0)}\hspace{-0.29cm}\frac{\left\|\xgrad d(\xi)\right\|^2}{\left\|\nabla d(\xi)\right\|}d\sigma(\xi)-\int_{\BmO}\hspace{-0.12cm}\frac{(\Lop_{A_\xi}d)(\xi^{-1}\circ x)}{\left(d(\xi^{-1}\circ x)\right)^{\alpha+1}}d\xi+\\
&+&(\alpha+1)\int_{\BmO}\frac{\left\langle A(x)\xgrad d(\xi^{-1}\circ x),\xgrad d(\xi^{-1}\circ x)\right\rangle}{\left(d(\xi^{-1}\circ x)\right)^{\alpha+2}}d\xi,
\end{eqnarray*}
where $A_\xi(x)=A(\xi\circ x)$. By recognizing that the term $\frac{\left\|\xgrad d(\xi)\right\|^2}{\left\|\nabla d(\xi)\right\|}=\psi_0(\xi)$ and using (\ref{somean}), we have
\begin{eqnarray*}
\Lop_Ah_\e(x)&\geq&\frac{\l}{\beta}(2K)^{Q-2-\alpha}-\int_{\BmO}\frac{(\Lop_{A_\xi}d)(\xi^{-1}\circ x)}{\left(d(\xi^{-1}\circ x)\right)^{\alpha+1}}d\xi+\\
&+&(\alpha+1)\int_{\BmO}\frac{\left\langle A(x)\xgrad d(\xi^{-1}\circ x),\xgrad d(\xi^{-1}\circ x)\right\rangle}{\left(d(\xi^{-1}\circ x)\right)^{\alpha+2}}d\xi.
\end{eqnarray*}
Exploiting (\ref{Lad}) and denoting by $\L_{A_x}=\max_{\left\|\xi\right\|=1}{\left\langle A(x)\xi,\xi\right\rangle}$, we deduce that
\begin{eqnarray*}
\Lop_A h_{\e}(x)\hspace{-0.18cm}&\geq&\frac{\l}{\beta}(2K)^{Q-2-\alpha}-2\L_{A_x}\int_{\BmO}\hspace{-1cm} \frac{\left\|V(\xi^{-1}\circ x)\right\|^2+\sum_{k=1}^n\left\|J_kV(\xi^{-1}\circ x)\right\|^2}{\left(d(\xi^{-1}\circ x)\right)^{\alpha+4}}d\xi\\
&+&\int_{\BmO}\hspace{-0.64cm}\frac{\left\langle \left((\alpha+4)A(x)-\Trac(A(x))\mathbb{I}_m\right)\xgrad d(\xi^{-1}\circ x),\xgrad d(\xi^{-1}\circ x)\right\rangle}{\left(d(\xi^{-1}\circ x)\right)^{\alpha+2}}d\xi\\
&=&\frac{\l}{\beta}(2K)^{Q-2-\alpha}-\int_{\BmO}\hspace{-0.27cm}\frac{\left(\vphantom{\sum_{k=1}^{N^2}}(2+2n)\L_{A_x}+\Trac(A(x))\right)\left\|v(\xi^{-1}\circ x)\right\|^2}{\left(d(\xi^{-1}\circ x)\right)^{\alpha+4}}d\xi\\
&+&(\alpha+4)\int_{\BmO}\hspace{-0.74cm}\frac{\left\langle A(x)\frac{\xgrad d(\xi^{-1}\circ x)}{\left\|\xgrad d(\xi^{-1}\circ x)\right\|},\frac{\xgrad d(\xi^{-1}\circ x)}{\left\|\xgrad d(\xi^{-1}\circ x)\right\|}\right\rangle\left\|v(\xi^{-1}\circ x)\right\|^2}{\left(d(\xi^{-1}\circ x)\right)^{\alpha+4}}d\xi,
\end{eqnarray*}
where in the last equality we have used the second property in (\ref{jayz}) and the orthonormality of the basis $X_1,\ldots,X_m,Z_1,\ldots,Z_n$. Assuming that $A$ satisfies the condition (\ref{deltaLandis}) for the $\delta$ we have fixed, then for any unit vector $\zeta$ we have
$$(\alpha+4)\left\langle A(x)\zeta,\zeta\right\rangle=(Q+4-\delta)\left\langle A(x)\zeta,\zeta\right\rangle\geq\Trac(A(x))+(2+2n)\L_{A_x}$$
uniformly in $x$. Therefore we finally get
$$\Lop_A h_\e(x)\geq\frac{\l}{\beta(2K)^{\alpha+2-Q}}=\frac{\l}{\beta(2K)^{2-\delta}}$$
for any $x\in O'$.
\end{proof}

We stress that the uniform convergence of $h_\e$ (and the resulting continuity of $h$) is given by the condition $\alpha<Q$. Moreover, as it is highlighted in \cite{GT}, for any $x\in\G$ we have
$$h(x)\geq-\frac{1}{\alpha}\int_{B_\rho(x)}\left(d(x^{-1}\circ\xi)\right)^{-\alpha}\,d\xi,$$
where $\rho$ is such that $\left|B_\rho\right|=\left|O\right|$. By the behavior of the Lebesgue measure under translations and dilations, for such $\rho$ we get
\begin{equation}\label{compmes}
0\geq h(x)\geq-\frac{\rho^{Q-\alpha}}{\alpha}\int_{\buz}\left(d(\xi)\right)^{-\alpha}\,d\xi=:-\gamma\left|O\right|^{1-\frac{\alpha}{Q}}.
\end{equation}
By keeping in mind the arguments in \cite{GT} (Theorem 3.2-3.3), we conclude the proof of Theorem \ref{cecd}.

\begin{proof}[Proof of Theorem \ref{cecd}] Fix $0<\l\leq\L$ and $A\in\mlL$ satisfying the $\delta$-Landis condition. Take $u\in\kA(B_2(0))$ and suppose there exists a point $\overline{x}$ in $B_{\frac{1}{2}}(0)$ where $u$ is less than $\frac{1}{2}$. We want to prove that 
\begin{equation}\label{want}
\left|\{x\in\buz\,:\,u(x)< 1\}\right|\geq \e\left|\buz\right|
\end{equation}
for some $0<\e<1$ depending just on $\l,\L,\delta$ and the structural constants of $\G$. Thus we will prove that $\kA$ satisfies the statement of the critical density property for $R=1$ and $x_0=0$. This is enough since for generic $R$ and $x_0$ we can exploit the left invariance and the homogeneity of the vector fields $X_j$'s. As a matter of fact, for any $v\in\kA(\dball)$ we can consider the function $u(x):=v(x_0\circ\delta_R(x))$ which belongs to $\kom^{\tilde{A}}(B_2(0))$ where $\tilde{A}(x)=A(\delta_{\frac{1}{R}}(x_0^{-1}\circ x))$. Therefore the critical density holds true with the same $\e$ and the same $c=\frac{1}{2}$ since we have $\left|\ball\right|=R^Q\left|\buz\right|$, 
$\left|\{x\in\ball\,:\,v(x)\geq 1\}\right|=R^Q\left|\{x\in \buz\,:\,u(x)\geq 1\}\right|$ and also $\inf_{B_{\frac{1}{2}}(0)}{u}=\inf_{\hball}{v}$.\\
In order to prove (\ref{want}), we use the barrier of the previous lemma and an auxiliar function involving $\phi=d^4$. In our notations, by (\ref{Hessphi}) we get
\begin{eqnarray*}
\Lop_A\phi(x)\hspace{-0.14cm}&=&\hspace{-0.2cm}4\left\|v(x)\right\|^2\Trac(A(x))+8\left\langle A(x)V(x),V(x)\right\rangle+8\sum_{k=1}^n\left\langle A(x)J_kV(x),J_kV(x)\right\rangle\\
&\leq&4\L(m+2+2n)\left\|v(x)\right\|^2\leq4(Q+2)\L
\end{eqnarray*}
for any $x\in\buz$. If $C$ is the positive constant in (\ref{compop}), we set
$$w(x)=\frac{C\l}{8(Q+2)\L}(u(x)+\phi(x)-1).$$
By the hypothesis $\Lop_A u\leq0$, hence we have $\Lop_A w\leq\frac{C}{2}\l$ in $\buz$. Moreover, $w$ is nonnegative on $\de\buz$ and
$$w(\overline{x})\leq\frac{C\l}{8(Q+2)\L}\left(\frac{1}{2}+\frac{1}{2^4}-1\right)=-\frac{7C}{128(Q+2)}\frac{\l}{\L}.$$
We put $O:=\{x\in\buz\,:\,w(x)<0\}$. We remark that $O$ is an open set such that $O\subseteq\{x\in\buz\,:\,u(x)< 1\}$. This set is non-empty since $\overline{x}\in O$. With this choice of $O$, we build the barrier $h$ of Lemma \ref{BAR} and we consider the continuous function $h-w$. We claim that $h-w\leq 0$ in $O$. We already know that this inequality holds true on $\de O$ since $w\geq 0$ on $\de\buz$. Suppose by contradiction that there exists $\xi_0\in O$ such that $h(\xi_0)-w(\xi_0)=2\delta>0$. Of course, this implies the existence of $\e_0>0$ such that $h_\e(\xi_0)-w(\xi_0)\geq\delta$ if $\e\leq\e_0$. Now, for any compact set $O'\subset O$ containing $\xi_0$, by Picone's maximum principle we would get $\max_{\de O'}{(h_\e-w)}\geq\delta$ if $\e<\min{\{\frac{1}{2}d(O',\de O),\e_0\}}$ since $\Lop_A(h_\e-w)\geq\frac{C}{2}\l$ in $O'$. Letting $\e\rightarrow0^+$, we deduce $\max_{\de O'}{(h-w)}\geq\delta$ for any $O'$ which is a contradiction since $h-w\leq 0$ on $\de O$. Thus we have proved the claim. In particular we get
$$-\frac{7C}{128(Q+2)}\frac{\l}{\L}\geq w(\overline{x})\geq h(\overline{x})\geq -\gamma\left|O\right|^{1-\frac{\alpha}{Q}}$$
by the relation (\ref{compmes}). Therefore we have
$$\left|\{x\in\buz\,:\,u(x)< 1\}\right|\geq\left|O\right|\geq\left(\frac{7C}{128\gamma(Q+2)}\frac{\l}{\L}\right)^{\frac{Q}{\delta}}=:\e\left|\buz\right|$$
and the theorem is proved.
\end{proof}

\section{Examples and further comments}\label{excom}

As remarked by Kaplan in \cite{K} (Corollary 1), there exist H-type algebras with center of any given dimension. We want to show here a representative class. In \cite{BLU} (Proposition 3.2.1) it is proved that a Lie group on $\RNN=\R^{m+n}$, which is homogeneous with respect to the dilations $\dl(x,t)=(\l x,\l^2 t)$ for $x\in\R^m$ and $t\in\RN$, is equipped with the composition law $\circ$ given by
\begin{equation}\label{grouptwo}
(x,t)\circ(x_1,t_1)=\left(x+x_1,t+t_1+\frac{1}{2}\left\langle Bx,x_1\right\rangle\right)\qquad\mbox{for }\, (x,t),(x_1,t_1)\in\RNN.
\end{equation}
Here we have denoted by $\left\langle Bx,x_1\right\rangle$ the vector of $\RN$ whose components are $\left\langle B^kx,x_1\right\rangle$ (for $k=1,\ldots,n$) for some suitable $m\times m$ matrices $B^1,\ldots,B^n$. According to \cite{BLU} (Definition 3.6.1), such a group is called prototype group of H-type if the following properties are satisfied:
\begin{itemize}
\item[$\cdot$] $B^j$ is skew-symmetric and orthogonal for any $j\leq n$;
\item[$\cdot$] $B^iB^j=-B^jB^i$ for every $i,j\in\{1,\ldots,n\}$ with $i\neq j$.
\end{itemize}
This class of homogeneous Lie groups belongs to the class of H-type groups and any H-type group is isomorphic to one of these (\cite{BLU}, Theorem 18.2.1). Consider the vector fields
$$X_j(x,t)=\de_{x_j}+\frac{1}{2}\sum_{k=1}^{n}(B^k x)_j\de_{t_k}\qquad\mbox{ for } j=1,\ldots,m.$$
In the literature, this basis is also called the Jacobian basis of $\g$ (for a definition we refer to \cite{BLU}, Proposition 1.2.16) because it comes from the Jacobian matrix of the left-translation. In this case, the standard inner product on $\g$ with respect to the basis $$X_1,\ldots,X_m,\frac{\de}{\de t_1}\ldots,\frac{\de}{\de t_n}$$
induces on $\g$ the structure of H-type algebra. The homogeneous norm $d$ we have exploited in the previous sections becomes here $\left\|x\right\|^4+16\left\|t\right\|^2$ (see e.g. \cite{BLU}, Remark 18.3.3).

\begin{ex} The Heisenberg-Weyl group is also a particular prototype group of H-type. As a matter of fact, the case of $m=2k$, $n=1$ and
$$B=B^1=\left(\begin{array}{cc}
0 & -\mathbb{I}_k  \\
\mathbb{I}_k & 0	
\end{array}\right)$$
turns out to be exactly the case of the Heisenberg group $\mathbb{H}^k$ in $\R^{2k+1}$. The vector fields $X_j$'s become the usual generators
$X_j=\de_{x_j}-\frac{x_{j+k}}{2}\de_{t}$ and $X_{j+k}=\de_{x_{j+k}}+\frac{x_j}{2}\de_{t}$, $j\in\{1,\ldots,k\}$. Other explicit examples of prototype groups (with $n\geq2$) are given in \cite{BLU} (Remark 3.6.6).
\end{ex}

We want to remark explicitly the importance of the orthonormality assumption of the horizontal basis $X_1,\ldots,X_m$ in our result. As a matter of fact, the condition (\ref{CordesLandis}) or the condition on the bound of $\frac{\L}{\l}$ is not stable under a change of the horizontal basis but it does depend on that choice. Furthermore, the orthonormality has been deeply used since this allows us to make the "right" choice of the gauge function related to the sub-Laplacian. Let us make an example. Let us consider the Lie group on $\R^5$ with the composition law as in (\ref{grouptwo}) and
$$B=B^1=\left(\begin{array}{cccc}
0 & -1 & 0 & 0 \\
1 & 0 & 0	& 0 \\
0 & 0 & 0	& -2 \\
0 & 0 & 2	&  0
\end{array}\right)$$
Here $m=4$ and $n=1$. This group is well-studied in the literature. Since this group is isomorphic to $\mathbb{H}^2$, it is in particular an H-type group but it is not a prototype one. In \cite{BT} (Example 6.6) Balogh and Tyson gave an explicit expression for the fundamental solution of the canonical sublaplacian $\Delta_{\G}=\sum_{j=1}^4X_j^2$, where $X_j=\de_{x_j}+\frac{1}{2}(B x)_j\de_{t}$ ($j=1,\ldots,4$) are the horizontal fields of the Jacobian basis. If we look at that formula, we can recognize it is different from being a power of $\left\|v(x)\right\|^4+16\left\|z(x)\right\|^2$. In \cite{B} Bonfiglioli proved that the gauge function associated to $\Delta_\G$ is not even horizontally convex. The problem is that the Jacobian basis is not orthonormal with respect to the scalar product inducing in this group the structure of H-type group. Thus, if we want to apply our result on the horizontally elliptic operator in this group $\Lop_A=\sum_{i,j=1}^ma_{ij}(x,t)X_iX_j$ we need that our Cordes-Landis condition is satisfied not for the matrix $A$ but for the matrix $D^tA(x,t)D$ where $D$ brings an orthonormal basis in $\{X_1,\ldots,X_m\}$. It is easy to verify that, in this situation, the basis $X_1, X_2, \frac{1}{\sqrt{2}}X_3, \frac{1}{\sqrt{2}}X_4$ is orthonormal.

\bibliographystyle{amsalph}

\end{document}